\documentclass{article}

\usepackage{amsmath, amssymb, amsfonts, amsbsy, amscd, latexsym, amsthm}

\date{}

\newtheorem{Theorem}{Theorem}[section]
\newtheorem{Proposition}[Theorem]{Proposition}

\newtheorem{Lemma}[Theorem]{Lemma}
\theoremstyle{definition}
\newtheorem{Definition}[Theorem]{Definition}

\numberwithin{equation}{section}

\title{Coxeter group actions on Saalsch\"utzian 
${}_4F_3(1)$ series and very-well-poised 
${}_7F_6(1)$ series}

\author{Ilia D. Mishev
\footnote{Department of
Mathematics, University of Colorado at Boulder,
Campus Box 395, Boulder, CO 80309-0395,
U.S.A.
E-mail address: ilia.mishev@colorado.edu}}

\begin{document}

\maketitle

\begin{abstract}

In this paper we consider a function
$L(\vec{x})=L(a,b,c,d;e;f,g)$, which can
be written as a linear combination of
two Saalsch\"utzian ${}_4F_3(1)$ 
hypergeometric series or as a 
very-well-poised ${}_7F_6(1)$
hypergeometric series.
We explore two-term and three-term relations
satisfied by the $L$ function and put
them in the framework of group theory. 

We prove a fundamental two-term
relation satisfied by the $L$ function
and show that 
this relation 
implies that the Coxeter 
group $W(D_5)$, 
which has 1920 elements, is an
invariance group for $L(\vec{x})$.
The invariance relations
for $L(\vec{x})$ 
are 
classified into six types based on a
double coset decomposition of the
invariance group. 
The fundamental two-term relation
is shown to generalize classical
results about hypergeometric series.
We derive
Thomae's identity for
${}_3F_2(1)$ series, Bailey's identity
for terminating Saalsch\"utzian
${}_4F_3(1)$ series, and Barnes' second
lemma as consequences.

We further explore three-term relations
satisfied by $L(a,b,c,d;e;f,g)$. 
The group that governs the three-term
relations is shown to be isomorphic to
the Coxeter group $W(D_6)$, which has
23040 elements.
Based on the right cosets of $W(D_5)$ in
$W(D_6)$, we demonstrate the existence
of 220 three-term relations satisfied by
the $L$ function that fall into two families
according to the notion of $L$-coherence. 

\end{abstract}

\section{Introduction}

Hypergeometric series were first introduced by Gauss
\cite{Ga}, who studied ${}_2F_1$ series. Generalized
hypergeometric series 
of type ${}_AF_B$, where $A$ and $B$ are positive
integers,
were studied in the late nineteenth
and early twentieth century 
by Thomae \cite{T}, Barnes \cite{Bar1, Bar2},
Ramanujan (see \cite{Har}),
Whipple \cite{Wh1, Wh3, Wh2},
Bailey \cite{Ba1, Ba}, and others. 

There has been a renewed interest in
hypergeometric series over the last
twenty-five years. Relations among hypergeometric
and basic hypergeometric series were put
into a group-theoretic framework in papers
by Beyer et al.\ \cite{BLS}, Srinivasa
Rao et al.\ \cite{Rao}, 
Formichella et al.\ \cite{FGS},
Van der Jeugt
and Srinivasa Rao \cite{V}, Livens
and Van der Jeugt \cite{LJ1, LJ2}.
Other papers include 
Groenevelt \cite{Groe},
van de Bult et al.\ \cite{BRS}, and
Krattenthaler and Rivoal \cite{KR}.

Hypergeometric series have
also appeared in recent papers by Bump \cite{Bu}, 
Stade \cite{St1, St2, St3, St4},
and Stade and Taggart \cite{ST}
with applications
in the theory of automorphic functions. Other 
recent works,
with applications in physics, were written by
Drake \cite{D}, Grozin \cite{Gr}, and
Raynal \cite{R}.

The goal of this paper is to describe
two-term and three-term relations among
Saalsch\"utzian 
${}_4F_3(1)$ hypergeometric series and
put them in the framework of group theory.
We examine a function
$L(a,b,c,d;e;f,g)$ 
(see (\ref{220}) for the definition)
which is a linear combination
of two Saalsch\"utzian ${}_4F_3(1)$ series.
This particular linear combination
of two Saalsch\"utzian ${}_4F_3(1)$ series
appears in \cite{St2} in the evaluation
of the Mellin transform of a spherical
principal series $GL(4,\mathbb{R})$
Whittaker function.

In Section 3 we derive a 
fundamental two-term relation 
(see (\ref{340})) satisfied by
$L(a,b,c,d;e;f,g)$. 
The fundamental two-term relation (\ref{340})
is derived through a Barnes integral
representation of $L(a,b,c,d;e;f,g)$ and
generalizes both Thomae's identity (see \cite[p.\ 14]{Ba})
and Bailey's
identity (see \cite[Eq.\ $(10.11)$]{Wh3} or
\cite[p.\ 56]{Ba}) in the sense that the latter
two identities can be obtained
as limiting cases of our
fundamental two-term relation (see Section 5).  

In Section 4 we show that 
the two-term relation (\ref{340})
combined with the trivial invariances
of $L(a,b,c,d;e;f,g)$ under permutations
of $a,b,c,d$ and interchanging
$f,g$ implies that
the function $L(a,b,c,d;e;f,g)$ has an
invariance group $G_L$ isomorphic
to the Coxeter group $W(D_5)$,
which is of order 1920.
(See \cite{H} for general information
on Coxeter groups.)
The invariance group $G_L$ is 
given as a matrix
group of transformations of the affine
hyperplane
\begin{equation}
\label{110}
V=\{(a,b,c,d,e,f,g)^T \in
\mathbb{C}^7:
e+f+g-a-b-c-d=1\}.
\end{equation}
The 1920 invariances
of the $L$ function that follow 
from the invariance group $G_L$ are 
classified into six types based
on a double coset decomposition of $G_L$
with respect to its subgroup $\Sigma$ 
consisting of all the permutation matrices
in $G_L$. To the best of the author's
knowledge, using such a double coset
decomposition is a new way of describing 
all the relations 
induced by an invariance group
and does not appear explicitly 
in the literature before.  

Some consequences of the fundamental
two-term relation (\ref{340}) are
shown in Section 5. In particular,
as already mentioned, we show that
Thomae's and Bailey's identities
follow as limiting cases of 
(\ref{340}). We also show that
Barnes' second lemma 
(see \cite{Bar2} or 
\cite[p.\ 42]{Ba})
follows as
a special case of (\ref{340})
when we take $d=g$.

In Section 6 we lay the group-theoretic
foundation for the three-term relations
satisfied by $L(a,b,c,d;e;f,g)$. The three-term
relations are governed by the right coset
space $G_L \backslash M_L$ described
in that section. 
We also introduce the 
notion of $L$-coherence (see Definition \ref{D650}),
which will be used in
the classification of the
three-term relations in the following section.

Section 7 describes the three-term relations satisfied
by $L(a,b,c,d;e;f,g)$. We show that
for every $\sigma_1,\sigma_2,\sigma_3 \in M_L$
such that $\sigma_1,\sigma_2,\sigma_3$ are in
different right cosets of $G_L$ in $M_L$,
there exists a three-term relation involving
the functions
$L(\sigma_1\vec{x}),L(\sigma_2\vec{x})$,
and $L(\sigma_3\vec{x})$.
Thus we demonstrate the existence of 220
three-term relations. The 220 three-term relations
are shown to fall into two families based
on the notion of $L$-coherence. 
We explicitly find a three-term relation of
each family in (\ref{710}) and
(\ref{740}), and then show that every
other three-term relation is obtained from
one of those two through 
a change of variable
of the form 
$\vec{x} \mapsto \mu \vec{x}$ applied 
to all terms and coefficients. 

Versions of the $L$ function 
(in terms of very-well-poised
${}_7F_6(1)$ series, see (\ref{230}))
were examined in the past
by Bailey \cite{Ba1},
Whipple \cite{Wh2}, 
and Raynal \cite{R}. 
The $L$ function appears as 
a Wilson function 
(a nonpolynomial extension 
of the Wilson polynomial) in
\cite{Gr}.
Van de Bult et al.\ \cite{BRS} examine
generalizations to
elliptic, hyperbolic, and trigonometric
hypergeometric functions.

A basic hypergeometric series
analog of the $L$ function (in terms of
${}_8\phi_7$ series) was studied by 
Van der Jeugt and Srinivasa Rao \cite{V}
and by Livens and Van der Jeugt \cite{LJ1}.
The authors establish an invariance group 
isomorphic to $W(D_5)$
for the ${}_8\phi_7$ series, 
but do not classify all two-term
relations, or consider how they 
could imply results about
lower-order series. They also do not
observe the idea of $L$-coherence
and how it divides the three-term relations
into two families.

Very recently
Formichella et al.\ \cite{FGS} explored
a function $K(a;b,c,d;e,f,g)$ which is
a different linear combination 
of two Saalsch\"utzian
${}_4F_3(1)$ series from 
the function $L(a,b,c,d;e;f,g)$. 
The particular linear combination
of 
${}_4F_3(1)$ series studied by
Formichella et al.\ appears in the
theory of archimedean zeta integrals
for automorphic $L$ functions 
(see \cite{St4,ST}).
The function 
$K(a;b,c,d;e,f,g)$ behaves very differently
from $L(a,b,c,d;e;f,g)$. 
Formichella et al.\ obtain in \cite{FGS}
a two-term relation
satisfied by $K(a;b,c,d;e,f,g)$ and show that
their two-term relation implies
that the symmetric group $S_6$ is an
invariance group for $K(a;b,c,d;e,f,g)$.
In addition, the existence of 4960
three-term relations satisfied by the
$K$ function is demonstrated and the
4960 three-term relations are classified
into five families based on 
the notion of Hamming type.
In a future work by the author of the
present paper and by Green and Stade,
the connection between the $K$ and the $L$
functions will be studied. 

{\it Acknowledgments.} This paper is based on results
obtained in the author's Ph.D. thesis (see \cite{M})
at the University of Colorado at Boulder. The author
would like to acknowledge the guidance of his advisor 
Eric Stade as well as the discussions with R.M. Green from
the University of Colorado at Boulder and 
Robert S. Maier from
the University of Arizona.
  
\section{Hypergeometric series and Barnes integrals}

The hypergeometric series of type
${}_{p+1}F_p$ is the power series
in the complex variable $z$ defined by
\begin{equation}
\label{210}
{}_{p+1}F_p\left[
{\displaystyle a_1,a_2,\ldots, a_{p+1};
\atop
\displaystyle b_1,b_2,\ldots, b_p;}z\right]=
\sum_{n=0}^{\infty}\frac
{(a_1)_n(a_2)_n\cdots (a_{p+1})_n}
{n!(b_1)_n(b_2)_n\cdots (b_p)_n}z^n,
\end{equation}
where $p$ is a positive integer,
the numerator parameters 
$a_1,a_2,\ldots, a_{p+1}$ and the
denominator parameters
$b_1,b_2,\ldots, b_p$ are complex numbers,
and the rising factorial 
$(a)_n$ is given by
\begin{equation*}
(a)_n=\left\{
\begin{array}{rl}
a(a+1)\cdots(a+n-1)=
\frac{\Gamma(a+n)}{\Gamma(a)}, & n>0,\\
1, & n=0.
\end{array} \right.
\end{equation*}

The series in (\ref{210}) converges absolutely
if $|z|<1$. When $|z|=1$, the series converges
absolutely if $\mbox{Re}(\sum_{i=1}^{p} b_i - 
\sum_{i=1}^{p+1} a_i) > 0$ (see \cite[p.\ 8]{Ba}).
We assume that no denominator parameter is
a negative integer or zero. If a numerator 
parameter is a negative integer or zero, the
series has only finitely many nonzero terms and
is said to terminate. 

When $z=1$, the series is
said to be of unit argument and of type
${}_{p+1}F_p(1)$. If
$\sum_{i=1}^{p} b_i = 
\sum_{i=1}^{p+1} a_i + 1$, the series is called 
Saalsch\"utzian.
If $1+a_1=b_1+a_2= \ldots = b_p+a_{p+1}$,
the series is called well-poised. A well-poised
series that satisfies $a_2=1+\frac{1}{2}a_1$ is
called very-well-poised.

Our main object of study in this paper will be 
the function $L(a,b,c,d;e;f,g)$ defined by
\begin{eqnarray}
\label{220}
&&L(a,b,c,d;e;f,g) \nonumber \\
&&=\frac{{}_4F_3\left[
{\displaystyle a,b,c,d;
\atop
\displaystyle e,f,g;}1\right]}{\sin \pi e\ \Gamma(e)
\Gamma(f)\Gamma(g)\Gamma(1+a-e)\Gamma(1+b-e)\Gamma(1+c-e)
\Gamma(1+d-e)} \nonumber \\
&&-\frac{{}_4F_3\left[
{\displaystyle 1+a-e,1+b-e,1+c-e,1+d-e;
\atop
\displaystyle 1+f-e,1+g-e,2-e;}1\right]}
{\sin \pi e\ \Gamma(a)
\Gamma(b)\Gamma(c)\Gamma(d)\Gamma(1+f-e)\Gamma(1+g-e)
\Gamma(2-e)},
\end{eqnarray}
where $a,b,c,d,e,f,g \in \mathbb{C}$ satisfy
$e+f+g-a-b-c-d=1$.

The function $L(a,b,c,d;e;f,g)$ is a linear combination
of two Saalsch\"utzian ${}_4F_3(1)$ series.
Other notations we will use for 
$L(a,b,c,d;e;f,g)$ are
$L\left[{\displaystyle a,b,c,d; \atop
\displaystyle e;f,g}\right]$ and
$L(\vec{x})$, where 
we will always have
$\vec{x}=(a,b,c,d,e,f,g)^{T} \in V$
(see (\ref{110})).

It should be noted that by 
\cite[Eq.\ (7.5.3)]{Ba},
the $L$ function can be expressed as
a very-well-poised 
${}_7F_6(1)$ series:
\begin{align}
\label{230}
&L(a,b,c,d;e;f,g) \nonumber\\
{}\\
&=\frac{\Gamma(1+d+g-e)}
{\pi \Gamma(g)\Gamma(1+g-e)
\Gamma(f-d)\Gamma(1+a+d-e)
\Gamma(1+b+d-e)\Gamma(1+c+d-e)}
\nonumber\\
&\cdot {}_7F_6\left[
{\displaystyle d+g-e,
1+\frac{1}{2}(d+g-e),g-a,g-b,g-c,
d,1+d-e;
\atop
\displaystyle \frac{1}{2}(d+g-e),
1+a+d-e,1+b+d-e,1+c+d-e,1+g-e,g;}
1\right],\nonumber
\end{align}
provided that $\mbox{Re}(f-d)>0$.
Therefore our results on the $L$
function can also be interpreted in terms
of the very-well-poised
${}_7F_6(1)$ series given in
(\ref{230}). 

Fundamental to the derivation
of a nontrivial two-term relation 
for the
$L$ function will be the notion
of a Barnes integral, which
is a contour integral of the
form
\begin{equation}
\label{240}
\int_t \prod_{i=1}^{n}\Gamma^{\epsilon_i}(a_i+t)
\prod_{j=1}^{m}\Gamma^{\epsilon_j}(b_j-t) \,dt,
\end{equation}
where $n,m \in \mathbb{Z}^{+};
\epsilon_i,\epsilon_j = \pm 1;$
and $a_i,b_j,t \in \mathbb{C}$.
The path of integration is
the imaginary axis,
indented if necessary, so that any poles
of $\prod_{i=1}^{n}\Gamma^{\epsilon_i}(a_i+t)$
are to the left of the contour and
any poles of
$\prod_{j=1}^{m}\Gamma^{\epsilon_j}(b_j-t)$ are to
the right of the contour. This path of integration
always exists, provided that, for
$1 \leq i \leq n$ and $1 \leq j \leq m$, we
have $a_i+b_j \notin \{0,-1,-2,\ldots\}$ whenever
$\epsilon_i=\epsilon_j=1$. 

From now on, when we write an integral
of the form (\ref{240}),
we will always
mean a Barnes integral with a path of integration
as just described.

A Barnes integral can often be evaluated in terms of
hypergeometric series using the
Residue Theorem,
provided that we can establish the
necessary convergence arguments.
This is the approach we
take in the next section. 
We will make use of the extension of
Stirling's formula to the complex numbers 
(see \cite[Section 4.42]{Titch} or \cite[Section 13.6]{WW}):
\begin{equation}
\label{250} \Gamma (a+z) = \sqrt{2\pi} z^{a+z-1/2}e^{-z} (1+
\mbox{O} (1/|z|))
\mbox{ uniformly as } |z| \to \infty,
\end{equation}
provided that 
$-\pi +\delta \leq \arg (z) \leq \pi -
\delta, \; \delta \in (0,\pi)$.

When applying the Residue Theorem,
we will use 
the fact
that the gamma 
function has simple poles at
$t=-n, n=0,1,2,\dots$, with
\begin{equation}
\label{260}
\mbox{Res}_{t=-n}\Gamma(t)=\frac{(-1)^n}{n!}.
\end{equation}

When simplifying expressions
involving gamma functions, 
the reflection formula for the
gamma function will
often be
used:
\begin{equation}
\label{270}
\Gamma(t)\Gamma(1-t)=
\frac{\pi}{\sin \pi t}.
\end{equation}

Finally, we will use
a result about Barnes integrals
known as Barnes' lemma
(see \cite{Bar1} or \cite[p.\ 6]{Ba}):

\begin{Lemma}[Barnes' lemma]
\label{BL}

If
$\alpha,\beta,\gamma,\delta \in \mathbb{C}$,
we have
\begin{eqnarray}
\label{290}
\frac{1}{2 \pi i} \int_t
\Gamma(\alpha+t) \Gamma(\beta+t)
\Gamma(\gamma-t) \Gamma(\delta-t) \,dt
\nonumber\\
=\frac{\Gamma(\alpha+\gamma)\Gamma(\alpha+\delta)
\Gamma(\beta+\gamma)\Gamma(\beta+\delta)}
{\Gamma(\alpha+\beta+\gamma+\delta)},
\end{eqnarray}
provided that none of 
$\alpha+\gamma,\alpha+\delta,
\beta+\gamma$, and $\beta+\delta$ belongs to
$\{0,-1,-2,\ldots\}$.

\end{Lemma}

\section{Fundamental two-term relation}

In this section we show that the function
$L(a,b,c,d;e;f,g)$ defined in (\ref{220}) can
be represented as a Barnes integral. The Barnes integral
representation will then be used to derive a
fundamental two-term relation satisfied by the $L$
function.

\begin{Proposition}
\label{P310}
\begin{eqnarray}
\label{310}
&&L(a,b,c,d;e;f,g) \nonumber \\
&&=\frac{1}{\pi \Gamma(a)\Gamma(b)\Gamma(c)\Gamma(d)
\Gamma(1+a-e)\Gamma(1+b-e)\Gamma(1+c-e)\Gamma(1+d-e)}
\nonumber \\
&&\cdot \frac{1}{2\pi i}
\int_t\frac{\Gamma(a+t)\Gamma(b+t)\Gamma(c+t)\Gamma(d+t)
\Gamma(1-e-t)\Gamma(-t)}{\Gamma(f+t)\Gamma(g+t)}\,dt.
\end{eqnarray}

\end{Proposition}

\begin{proof}

Let
\begin{eqnarray}
\label{330}
&&I\left[
{\displaystyle a,b,c,d;
\atop
\displaystyle e;f,g}\right]\nonumber\\
&&=\frac{1}{2\pi i}
\int_t\frac{\Gamma(a+t)\Gamma(b+t)\Gamma(c+t)\Gamma(d+t)
\Gamma(1-e-t)\Gamma(-t)}{\Gamma(f+t)\Gamma(g+t)}\,dt.
\end{eqnarray}

For $N \geq 1$, let $C_N$ be the semicircle of radius $\rho_N  $ on
the right side of the imaginary axis and center at the origin, chosen
in such a way that $\rho_N \to \infty$ as $N \to \infty$ and
$$
\varepsilon:= \inf_N  
\mbox{dist}(C_N, \mathbb{Z}\cup (\mathbb{Z}-e)) >0.
$$

The formula
(\ref{270}) gives
\begin{eqnarray*}
G(t):=\frac{\Gamma(a+t)\Gamma(b+t)\Gamma(c+t)\Gamma(d+t)
\Gamma(1-e-t)\Gamma(-t)}{\Gamma(f+t)\Gamma(g+t)}\\
=\frac{-\pi^2\Gamma(a+t)\Gamma(b+t)
\Gamma(c+t)\Gamma(d+t)}{\Gamma(f+t)\Gamma(g+t)
\Gamma(e+t)\Gamma(1+t)\sin \pi t \sin \pi (e+t)}.
\end{eqnarray*}

By Stirling's formula (\ref{250}),
$$\frac{\Gamma(a+t)\Gamma(b+t)
\Gamma(c+t)\Gamma(d+t)}{\Gamma(f+t)\Gamma(g+t)
\Gamma(e+t)\Gamma(1+t)} \sim
t^{a+b+c+d-e-f-g-1}=t^{-2}.$$

It follows that we can find a constant $K>0$ such that
$$ |G(t)| \leq K/|t|^2 \quad \text{if} \; \; t \in C_N, \;\; N=1,2,\ldots, $$
which implies
$$\int_{C_N} G(t) \,dt \to 0 \quad \mbox{as } N \to \infty.$$

Therefore the integral given by
$I\left[
{\displaystyle a,b,c,d;
\atop
\displaystyle e;f,g}\right]$ 
is equal to the sum of the
residues of the 
integrand at the poles of
$\Gamma(1-e-t)$ and $\Gamma(-t)$, which
yields the result.

\end{proof}

The fundamental two-term relation satisfied
by $L(a,b,c,d;e;f,g)$ is given in the
next proposition.

\begin{Proposition}
\label{P330}

\begin{equation}
\label{340}
L\left[
{\displaystyle a,b,c,d;
\atop
\displaystyle e;f,g}\right]
=L\left[
{\displaystyle a,b,g-c,g-d;
\atop 
\displaystyle 1+a+b-f;1+a+b-e,g}\right].
\end{equation}

\end{Proposition}

\begin{proof}

Let
$I\left[
{\displaystyle a,b,c,d;
\atop
\displaystyle e;f,g}\right]$
be as given in (\ref{330}). As a first step,
we will prove that 
\begin{eqnarray}
\label{350}
\frac{I\left[
{\displaystyle a,b,c,d;
\atop
\displaystyle e;f,g}\right]}{\Gamma(c)\Gamma(d)
\Gamma(1+a-e)\Gamma(1+b-e)}\nonumber\\
=\frac{I\left[
{\displaystyle a,b,g-c,g-d;
\atop
\displaystyle 1+a+b-f;1+a+b-e,g}\right]}{\Gamma(f-a)\Gamma(f-b)
\Gamma(g-c)\Gamma(g-d)}.
\end{eqnarray}

By Barnes' lemma,
\begin{align*}
&\frac{\Gamma(a+t)\Gamma(b+t)}
{\Gamma(f+t)}\\
=\frac{1}{2 \pi i\Gamma(f-a)\Gamma(f-b)}
\int_{u}&\Gamma(t+u)\Gamma(f-a-b+u)
\Gamma(a-u)\Gamma(b-u)\,du
\end{align*}
and
\begin{align*}
&\frac{\Gamma(c+t)\Gamma(d+t)}{\Gamma(g+t)}\\
=\frac{1}{2 \pi i\Gamma(g-c)\Gamma(g-d)}
\int_{v}&\Gamma(t+v)\Gamma(g-c-d+v)
\Gamma(c-v)\Gamma(d-v)\,dv.
\end{align*}

We re-write the integral for
$I\left[
{\displaystyle a,b,c,d;
\atop
\displaystyle e;f,g}\right]$ by substituting for the above expressions, 
changing the order of integration, so that we integrate 
with respect to $t$ first
(the change in the order of integration
is readily justified using Stirling's formula
and Fubini's theorem), 
and then applying Barnes' Lemma again 
to the integral with respect to $t$. We obtain
\begin{eqnarray}
\label{360}
&&\frac{I\left[
{\displaystyle a,b,c,d;
\atop
\displaystyle e;f,g}\right]}{\Gamma(c)\Gamma(d)
\Gamma(1+a-e)\Gamma(1+b-e)}\nonumber\\
&&=\frac{-1}
{4 \pi^2 \Gamma(c)\Gamma(d)
\Gamma(1+a-e)\Gamma(1+b-e)
\Gamma(f-a)\Gamma(f-b)
\Gamma(g-c)\Gamma(g-d)}\nonumber\\
&&\cdot
\int_{u} \Gamma(f-a-b+u)\Gamma(a-u)
\Gamma(b-u)\Gamma(u)\Gamma(1-e+u)\nonumber\\
&&\cdot \left( \int_{v} \frac
{\Gamma(g-c-d+v)\Gamma(c-v)
\Gamma(d-v)\Gamma(v)\Gamma(1-e+v)}
{\Gamma(1-e+u+v)}\,dv \right) du.
\end{eqnarray}

After the substitution 
$v \mapsto c+d-f+v$ in the inside integral,
it is easily checked 
(using the Saalsch\"utzian
condition
$e+f+g-a-b-c-d=1$)
that the
right-hand side of (\ref{360})
is invariant under the transformation
$$(a,b,c,d;e;f,g) \mapsto
(a,b,g-c,g-d;1+a+b-f;1+a+b-e,g),$$
which proves (\ref{350}). The result
in the proposition now follows 
immediately from 
(\ref{350}) upon writing the two $L$ functions
in (\ref{340}) in terms of their
Barnes integral representations (\ref{310}).

\end{proof}

\section{Invariance group}

In the previous section we showed that
the function $L(a,b,c,d;e;f,g)$ satisfies 
the two-term relation (\ref{340}). If we
define
\begin{equation}
\label{410}
A=\begin{pmatrix}
1 & 0 & 0 & 0 & 0 & 0 & 0\\
0 & 1 & 0 & 0 & 0 & 0 & 0\\
0 & 0 & -1 & 0 & 0 & 0 & 1\\
0 & 0 & 0 & -1 & 0 & 0 & 1\\
0 & 0 & -1 & -1 & 1 & 0 & 1\\
0 & 0 & -1 & -1 & 0 & 1 & 1\\
0 & 0 & 0 & 0 & 0 & 0 & 1\\
\end{pmatrix} 
\in GL(7,\mathbb{C}),
\end{equation}
then (\ref{340}) can be
expressed as
$L(\vec{x})=L(A\vec{x})$.

If $\sigma \in S_7$, we will
identify $\sigma$ with the matrix
in $GL(7,\mathbb{C})$ that
permutes the standard basis
$\{e_1,e_2,\ldots,e_7\}$ of the 
complex vector
space $\mathbb{C}^7$
according to the permutation
$\sigma$. For example,
$$(123)= \begin{pmatrix}
0 & 0 & 1 & 0 & 0 & 0 & 0\\
1 & 0 & 0 & 0 & 0 & 0 & 0\\
0 & 1 & 0 & 0 & 0 & 0 & 0\\
0 & 0 & 0 & 1 & 0 & 0 & 0\\
0 & 0 & 0 & 0 & 1 & 0 & 0\\
0 & 0 & 0 & 0 & 0 & 1 & 0\\
0 & 0 & 0 & 0 & 0 & 0 & 1\\
\end{pmatrix}.$$

Let
\begin{equation}
\label{420}
G_L=\langle (12),(23),
(34),(67),A \rangle
\leq GL(7,\mathbb{C}).
\end{equation}
The two-term relation
(\ref{340}) along 
with the trivial
invariances of 
the function
$L(a,b,c,d;e;f,g)$
under permutations of $a,b,c,d$ and
interchanging $f,g$ implies that
$G_L$ is an invariance group
for $L(a,b,c,d;e;f,g)$, i.e.
$L(\vec{x})=L(\alpha \vec{x})$
for every $\alpha \in G_L$.

The goal of this section is to
find the isomorphism type of
the group $G_L$ and further to 
describe the two-term relations
for the $L$ function in terms of
a double coset decomposition
of $G_L$ with respect to its
subgroup $\Sigma$ defined as
follows:
\begin{equation}
\label{430}
\Sigma=\langle (12),(23),
(34),(67) \rangle.
\end{equation}
The group $\Sigma$ is a
subgroup of $G_L$ consisting
of permutation matrices. It is
clear that $\Sigma \cong S_4 \times S_2$
and so $|\Sigma|=48$.
We note that if $\sigma \in \Sigma,
\alpha \in G_L$, the multiplication
$\sigma\alpha$ permutes the rows of
$\alpha$, and the multiplication
$\alpha\sigma$ permutes the columns of
$\alpha$.
A double coset of 
$\Sigma$ in $G_L$ 
is a set of the form
\begin{equation}
\label{440}
\Sigma \alpha \Sigma=
\{\sigma \alpha \tau : \sigma,\tau \in \Sigma\},
\mbox{ for some } \alpha \in G_L.
\end{equation}
The distinct double cosets
of the form (\ref{440}) partition
the group $G_L$ and give us
a double coset decomposition of
$G_L$ with respect to $\Sigma$. (See 
\cite[p. 119]{DF} for
more on double cosets.)

In Theorem \ref{T410} below we show
that the group $G_L$ is isomorphic
to the Coxeter group $W(D_5)$,
which is of
order 1920. In Theorem \ref{T420}
we show that the subgroup $\Sigma$
is the largest permutation subgroup
of $G_L$ and obtain a double coset
decomposition of $G_L$ with respect to
$\Sigma$.  
We list a representative for
each of the six double cosets obtained
and give the six invariance relations
induced by those representatives
(see (\ref{450})--(\ref{460})). 
The six invariance relations 
(\ref{450})--(\ref{460})
listed are
all the ``different" types of invariance
relations in the sense that every other invariance
relation can be obtained by permuting
the first four entries and 
permuting
the last two entries 
on the right-hand side
of a listed invariance relation 
(which corresponds
to permuting the rows of the accompanying matrix), and
by permuting $a,b,c,d$ and  permuting $f,g$
on the right-hand side
of a listed invariance relation 
(which corresponds to permuting the
columns of the accompanying matrix).
  
\begin{Theorem}
\label{T410}

The group $G_L$ is isomorphic to the Coxeter group 
$W(D_5)$, which is of order 1920.

\end{Theorem}

\begin{proof}

The Dynkin diagram of the Coxeter
group $W(D_n)$ is given by the graph
with vertices labeled 
$1',1,2,\ldots,n$, where
$i,j \in \{1,2,\ldots,n\}$ are connected
by an edge if 
and only if $|i-j|=1$, and
$1'$ is connected to $2$ only.
The presentation of $W(D_n)$ is
given by
$$W(D_n)=\langle s_{1'},s_1,s_2,\ldots,s_n:
(s_is_j)^{m_{ij}}=1 \rangle,$$
where
$m_{ii}=1$ for all $i$; and 
for $i$ and $j$ distinct,
$m_{ij}=3$ if $i$ and $j$ are
connected by an edge, and 
$m_{ij}$=2 otherwise.
It is well-known that the order of $W(D_n)$ is
$2^{n-1}n!$ (see \cite[Section $2.11$]{H}). 

Consider the generators of $G_L$ given by
\begin{equation}
\label{470}
a_1=(34),a_2=(23),a_3=(34)A,
a_4=(67), a_{1'}=(12).
\end{equation}
A direct computation shows that 
$$(a_ia_j)^{m_{ij}}=1, \quad \mbox{for all }
i,j \in \{1,2,3,4,1'\}.$$

Therefore if we define $\varphi(s_i)=a_i$ for 
every $i \in \{1,2,3,4,1'\}$, $\varphi$
extends (uniquely) to a surjective homomorphism
from $W(D_5)$ onto $G_L$. 
Since $W(D_5)$ is a finite
group, we need to show that $G_L$ and $W(D_5)$ have
the same order to complete the proof. To that end, 
it is enough to just show that 
$|G_L| > 960=\frac{|W(D_5)|}{2}$.
An estimate on the order of $G_L$ is obtained
by computing the sizes of the distinct double cosets
$\Sigma A \Sigma$ and 
$\Sigma [(123)(67)A]^2 \Sigma$
of $\Sigma$ in $G_L$,
where $\Sigma$ is as given in (\ref{430}).
Indeed, by permuting columns that are different as
multisets and then permuting the rows of the resulting 
matrices
in every possible way, we see that
the number of matrices that
belong to each of those
double cosets 
is $12 \cdot 48$, which
shows that $|G_L|>960$ and
completes the proof. 

\end{proof}

As stated before Theorem \ref{T410},
we are interested in the complete
double coset decomposition of $G_L$
with respect to $\Sigma$ since this
will classify all the invariance
relations for the function
$L(a,b,c,d;e;f,g)$ in a convenient way.
We use the same  
technique as in the proof of
Theorem \ref{T410} 
given by permuting columns that are different as
multisets and then permuting the rows of the resulting 
matrices
in every possible way. 
We find that there are
six double cosets of $\Sigma$ in $G_L$.
Representative matrices for the double cosets are
$I_7, A, [(123)(67)A]^2, 
[(123)(67)A]^3, [(123)A]^3,
[(123)(67)A]^4$. The corresponding
double coset sizes are 
$1 \cdot 48, 12 \cdot 48, 12 \cdot 48,
12 \cdot 48, 2 \cdot 48, 1 \cdot 48$.
Furthermore, the representative
matrices are all seen to have different entries,  
so that
$\Sigma$ must indeed be the largest permutation
subgroup of $G_L$. 
Each representative matrix gives rise to
an invariance relation. 
Theorem \ref{T420} summarizes the result.
 
\begin{Theorem}
\label{T420}

Let 
$\Sigma$ be as defined in (\ref{430}). Then
$\Sigma$ consists of all the permutation matrices
in $G_L$. There are six double
cosets in the double coset decomposition
of $G_L$ with respect to $\Sigma$.
Representative matrices for the double cosets
are $I_7, A, [(123)(67)A]^2, 
[(123)(67)A]^3, [(123)A]^3,
[(123)(67)A]^4$ and the corresponding
double coset sizes are
$1 \cdot 48, 12 \cdot 48, 12 \cdot 48,
12 \cdot 48, 2 \cdot 48, 1 \cdot 48$.
The corresponding invariances of the $L$ function are
given by
\begin{align}
L\left[
{\displaystyle a,b,c,d;
\atop
\displaystyle e;f,g}\right] &=
L\left[
{\displaystyle a,b,c,d;
\atop
\displaystyle e;f,g}\right],
\label{450}\\
L\left[
{\displaystyle a,b,c,d;
\atop
\displaystyle e;f,g}\right] &=
L\left[
{\displaystyle a,b,g-c,g-d;
\atop
\displaystyle 1+a+b-f;1+a+b-e,g}\right],
\label{452}\\
L\left[
{\displaystyle a,b,c,d;
\atop
\displaystyle e;f,g}\right] &=
L\left[
{\displaystyle 1+a-e,g-c,a,f-c;
\atop
\displaystyle 1+a-c;1+a+b-e,1+a+d-e}\right],
\label{454}\\
L\left[
{\displaystyle a,b,c,d;
\atop
\displaystyle e;f,g}\right] &=
L\left[
{\displaystyle 1+d-e,1+a-e,g-c,g-b;
\atop
\displaystyle 1+g-b-c;1+a+d-e,1+g-e}\right],
\label{456}\\
L\left[
{\displaystyle a,b,c,d;
\atop
\displaystyle e;f,g}\right] &=
L\left[
{\displaystyle g-a,g-b,g-c,g-d;
\atop
\displaystyle 1+g-f;1+g-e,g}\right],
\label{458}\\
L\left[
{\displaystyle a,b,c,d;
\atop
\displaystyle e;f,g}\right] &=
L\left[
{\displaystyle 1+c-e,1+d-e,1+a-e,1+b-e;
\atop
\displaystyle 2-e;1+g-e,1+f-e}\right].
\label{460}
\end{align}

\end{Theorem}

\section{Applications of the
fundamental two-term relation}

In this section we prove
some consequences of the fundamental
two-term relation given in
Proposition \ref{P330}. As a first
step, we write 
the two $L$ functions
in (\ref{340}) in terms of
their definitions 
as linear combinations of two
${}_4F_3(1)$ series. 
We obtain
\begin{eqnarray}
\label{510}
\frac{{}_4F_3\left[
{\displaystyle a,b,c,d;
\atop
\displaystyle e,f,g;}1\right]}
{\sin \pi e\ \Gamma(e)\Gamma(f)
\Gamma(g)\Gamma(1+a-e)
\Gamma(1+b-e)\Gamma(1+c-e)
\Gamma(1+d-e)}\nonumber\\
-\frac{{}_4F_3\left[
{\displaystyle 1+a-e,1+b-e,
1+c-e,1+d-e;
\atop
\displaystyle 1+f-e,
1+g-e,2-e;}1\right]}
{\sin \pi e\ \Gamma(a)\Gamma(b)
\Gamma(c)\Gamma(d)\Gamma(1+f-e)
\Gamma(1+g-e)\Gamma(2-e)}\nonumber\\
=\frac{{}_4F_3\left[
{\displaystyle a,b,g-c,g-d;
\atop
\displaystyle 1+a+b-f,
1+a+b-e,g;}1\right]}
{\left[
{\displaystyle \sin \pi (1+a+b-f) 
\Gamma(1+a+b-f)
\Gamma(1+a+b-e)\Gamma(g)
\atop
\displaystyle \cdot
\Gamma(f-b)\Gamma(f-a)
\Gamma(1+d-e)\Gamma(1+c-e)
}\right]
}\nonumber\\
-\frac{{}_4F_3\left[
{\displaystyle f-b, f-a, 1+d-e, 1+c-e;
\atop
\displaystyle 1+f-e,f+g-a-b,1+f-a-b}1
\right]}
{\left[
{\displaystyle \sin \pi (1+a+b-f) 
\Gamma(a)\Gamma(b)
\Gamma(g-c)\Gamma(g-d)
\atop
\displaystyle \cdot
\Gamma(1+f-e)
\Gamma(f+g-a-b)\Gamma(1+f-a-b)
}\right]
}.
\end{eqnarray}

We fix $b,c,d,f,g \in
\mathbb{C}$ in such a way that
\begin{equation}
\label{520}
\mbox{Re}(f+g-b-c-d)>0, \quad
\mbox{Re}(f-b)>0.
\end{equation}
Let $a \in \mathbb{C}$ and 
let $e=1+a+b+c+d-f-g$ depend
on $a$. In equation 
(\ref{510}) we let
$|a| \to \infty$. Using Stirling's 
formula (\ref{250}) and the
conditions (\ref{520}), we obtain
\begin{eqnarray}
\label{530}
\frac{{}_3F_2\left[
{\displaystyle b,c,d;
\atop
\displaystyle f,g;}1\right]}
{\Gamma(f)\Gamma(g)
\Gamma(f+g-b-c-d)}
\nonumber\\
=\frac{{}_3F_2\left[
{\displaystyle b,g-c,g-d;
\atop
\displaystyle f+g-c-d,g;}
1\right]}
{\Gamma(f+g-c-d)\Gamma(g)
\Gamma(f-b)}.
\end{eqnarray}
We note that the conditions
(\ref{520}) are needed for
the absolute convergence of
the two ${}_3F_2(1)$ series
in (\ref{530}).
Applying
(\ref{530}) twice yields 
Thomae's identity
\begin{eqnarray}
\label{540}
\frac{{}_3F_2\left[
{\displaystyle b,c,d;
\atop
\displaystyle f,g;}1\right]}
{\Gamma(f)\Gamma(g)
\Gamma(f+g-b-c-d)}
\nonumber\\
=\frac{{}_3F_2\left[
{\displaystyle f-b,g-b,f+g-b-c-d;
\atop
\displaystyle f+g-b-d,f+g-b-c;}
1\right]}
{\Gamma(b)\Gamma(f+g-b-d)
\Gamma(f+g-b-c)}.
\end{eqnarray}
In fact, applying
(\ref{540}) twice gives
(\ref{530}), so that
(\ref{530}) and (\ref{540})
are equivalent.

Next in equation (\ref{510}) we let
$a \to -n$, where $n$ is a
nonnegative integer. Using the fact
that $\lim_{a \to -n}\frac{1}{\Gamma(a)}
=0$ and then formula (\ref{270}) to
simplify the result, we obtain
Bailey's identity
\begin{align}
\label{550}
&{}_4F_3\left[
{\displaystyle -n,b,c,d;
\atop
\displaystyle e,f,g;}
1\right]\nonumber\\
=\frac{(e-b)_n(f-b)_n}
{(e)_n(f)_n}
&{}_4F_3\left[
{\displaystyle -n,b,g-c,g-d;
\atop
\displaystyle 1-n+b-f,1-n+b-e,g;}
1\right], 
\end{align}
which holds provided that
$e+f+g-b-c-d+n=1$.

Thomae's and Bailey's identities
have been shown 
in \cite{FGS}
in a similar way
to be limiting cases of a fundamental
two-term relation 
satisfied by the
function $K(a;b,c,d;e,f,g)$. 

As a final application, in the 
fundamental two-term relation
(\ref{340}) we let $d=g$.
We express the left-hand side
as a Barnes integral 
according to Proposition \ref{310},
and we write
the right-hand side in terms 
of two ${}_4F_3(1)$ series
according to the 
definition
of the $L$ function.
The condition $d=g$ causes
one of the terms on the right-hand 
side to go to zero and the
${}_4F_3(1)$ series in the other
term to be trivially equal to one.
If we simplify the result further using
(\ref{270}),
we obtain
\begin{eqnarray}
\label{560}
\frac{1}{2 \pi i}
\int_t\frac{\Gamma(a+t)\Gamma(b+t)
\Gamma(c+t)\Gamma(1-e-t)\Gamma(-t)}
{\Gamma(f+t)}\,dt\nonumber\\
=\frac{\Gamma(a)\Gamma(b)\Gamma(c)
\Gamma(1+a-e)\Gamma(1+b-e)\Gamma(1+c-e)}
{\Gamma(f-a)\Gamma(f-b)\Gamma(f-c)},
\end{eqnarray}
which holds provided that 
$e+f-a-b-c=1$. The equation (\ref{560})
is precisely 
the statement of Barnes' second lemma.

\section{Group-theoretic structure of the
three-term relations}

Let 
\begin{equation}
\label{610}
M_L=\langle (12),(23),(34),(56),(67),A \rangle
\leq GL(7,\mathbb{C}).
\end{equation}
The group $M_L$ is generated by the invariance
group $G_L$ and the transposition matrix $(56)$.
We will show that the right cosets of $G_L$ in $M_L$
govern the three-term relations for the function
$L(a,b,c,d;e;f,g)$. In particular, 
in the next section we will show that
for every $\sigma_1,\sigma_2,\sigma_3 \in M_L$
such that $\sigma_1,\sigma_2,\sigma_3$ are in
different right cosets of $G_L$ in $M_L$,
there exists a three-term relation involving
the functions
$L(\sigma_1\vec{x}),L(\sigma_2\vec{x})$,
and $L(\sigma_3\vec{x})$.

In this section we lay the foundation
of the group-theoretic structure of the three-term
relations satisfied by the $L$ function
by studying the right coset space
$G_L \backslash M_L$ and then introducing the
notion of $L$-coherence (see Definition \ref{D650}),
which will lead us to the classification of the
three-term relations through Proposition \ref{P660}.

\begin{Theorem}

\label{T610}

The group $M_L$ is isomorphic to the
Coxeter group $W(D_6)$, which has
23040 elements.

\end{Theorem}

\begin{proof}

The proof follows the same lines as the
proof of Theorem \ref{T410}. 
We let $a_1, a_2, a_3, a_4,a_{1'}$ be as
in (\ref{470}), and we also
let $a_5=(56)$. We obtain a surjective
homomorphism from $W(D_6)$ onto $M_L$,
which homomorphism is then shown to
be an isomorphism by estimating the
order of $M_L$ using the counting
technique from the proof of Theorem \ref{T410}.

\end{proof}

There are $23040 / 1920 = 12$ right cosets
of $G_L$ in $M_L$. The group $M_L$ acts by
right multiplication on those right cosets
through
$$(G_L \mu) \cdot \nu = G_L (\mu\nu), \quad
\mbox{for } \mu,\nu \in M_L.$$

It is well-know that the Coxeter group $W(D_6)$
has a center consisting of two elements
(see \cite[pp.\ 82 and 132]{H}). 
We let $w_0$ denote
the unique nonidentity element in the center
of $M_L$. The element $w_0$ is called the central involution.
The following proposition follows from standard facts about
the Coxeter group $W(D_n)$ (see \cite[Section 2.10]{H}):

\begin{Proposition}
\label{P620}

We can label the $12$ right cosets of $G_L$ in $M_L$
by $$1,2, \ldots, 6, \bar{1}, \bar{2}, \ldots, \bar{6}$$
in such a way that $i$ and $\bar{i}$ are interchanged
under the action of $w_0$,
for every $i \in \{1,2,\ldots,6\}$.

\end{Proposition}  

Indeed, the central involution $w_0$ is
computed to be
$$w_0=(12)(34)
[[(1234)(567)]^2A]^4,$$
and from here, representatives
of the twelve right cosets
$6,\dots,1,\bar{6},\ldots,\bar{1}$ 
are computed to be
$$\mu_6=I_7, \quad \mu_5=(56), \quad 
\mu_4=(57),$$
$$\mu_3=w_2, \quad \mu_2=(56)w_2, \quad 
\mu_1=(57)w_2,$$
$$\mu_{\bar{6}}=w_0, \quad \mu_{\bar{5}}=(56)w_0, \quad 
\mu_{\bar{4}}=(57)w_0,$$
$$\mu_{\bar{3}}=w_1, 
\quad \mu_{\bar{2}}=(56)w_1, 
\quad \mu_{\bar{1}}=(57)w_1,$$
respectively,
where
$$w_1=(1234)
[(1234)(567)A]^3(1432),$$
$$w_2=w_0w_1.$$

Every matrix $\mu \in M_L$ induces a permutation of the
twelve element set 
$\{1, \ldots , 6, \bar{1}, \ldots ,
\bar{6} \}$.
Let $\Phi:M_L \to S_{12}$ be the induced permutation
representation.
We note that if $\mu \in M_L$, then
the permutation $\Phi(\mu)$ can be uniquely
described by specifying its effect on the
set $\{1,\ldots,6\}$, since the elements
in $\{\bar{1},\ldots,\bar{6}\}$ will be
permuted based on where $\{1,\ldots,6\}$
are permuted with the addition or 
omission of a bar.
We give the
images under $\Phi$ 
of the generators of $M_L$ in 
the following proposition: 

\begin{Proposition}
\label{P630}

The images of the generators
$a_1,a_2,a_3,a_4,a_5$, and $a_{1'}$ of $M_L$ under
the permutation representation 
$\Phi:M_L \to S_{12}$ are given by
\begin{align*}
\Phi(a_1)=\Phi((34))=
\left( \begin{matrix}
1 & 2 & 3 & 4 & 5 & 6\\
2 & 1 & 3 & 4 & 5 & 6\\
\end{matrix} \right),\\
\Phi(a_2)=\Phi((23))= 
\left( \begin{matrix}
1 & 2 & 3 & 4 & 5 & 6\\
1 & 3 & 2 & 4 & 5 & 6\\
\end{matrix} \right),\\
\Phi(a_3)=\Phi((34)A)= 
\left( \begin{matrix}
1 & 2 & 3 & 4 & 5 & 6\\
1 & 2 & 4 & 3 & 5 & 6\\
\end{matrix} \right),\\
\Phi(a_4)=\Phi((67))= 
\left( \begin{matrix}
1 & 2 & 3 & 4 & 5 & 6\\
1 & 2 & 3 & 5 & 4 & 6\\
\end{matrix} \right),\\
\Phi(a_5)=\Phi((56))= 
\left( \begin{matrix}
1 & 2 & 3 & 4 & 5 & 6\\
1 & 2 & 3 & 4 & 6 & 5\\
\end{matrix} \right),\\
\Phi(a_{1'})=\Phi((12))= 
\left( \begin{matrix}
1 & 2 & 3 & 4 & 5 & 6\\
\bar{2} & \bar{1} & 3 & 4 & 5 & 6
\end{matrix} \right).
\end{align*}

\end{Proposition}

Given the decomposition
of an element $\mu \in M_L$ in
terms of the generators
$a_1,a_2,a_3,a_4,a_5,a_{1'}$,
Proposition (\ref{P630}) can be
used to find the permutation
corresponding to $\Phi(\mu)$.

The next proposition is a 
restatement of standard facts about
the construction of the
Coxeter group $W(D_n)$ as
a semidirect product
(see \cite[Section 2.10]{H} for
more details).

\begin{Proposition}
\label{P640}

The permutation representation $\Phi:M_L \to S_{12}$
is faithful, i.e. $\emph{ker}(\Phi)=\{I_7\}$.
The embedding of $M_L$ into $S_{12}$ consists of all
permutations of the form 
$\left( \begin{matrix}
1 & 2 & 3 & 4 & 5 & 6\\
j_1 & j_2 & j_3 & j_4 & j_5 & j_6\\
\end{matrix} \right)$, where each
$j_i$ belongs to the set 
$\{1, \ldots , 6,\bar{1}, \ldots , \bar{6}\}$,
the $j_i$'s are all distinct if we remove the bars,
and an even number, i.e. 0, 2, 4, or 6, of the
$j_i$'s contain a bar.

\end{Proposition}

Central to the classification of the
three-term relations in the next section is
the notion of $L$-coherence defined as follows:

\begin{Definition}
\label{D650}

If $S$ is a subset of
$G_L \backslash M_L$, we say that $S$ is $L$-coherent
if no two elements of $S$ are interchanged by the action
of the central involution $w_0$. If $S$ is not
$L$-coherent, then $S$ is called
$L$-incoherent. 

\end{Definition}

It is clear that a set $S$ is $L$-coherent
if and only if $S$ does not contain {\it both} 
elements of the
form $i$ and $\bar{i}$, for any 
$i \in \{1, \ldots , 6\}$.

The group $M_L$ acts by right multiplication 
on triples of cosets of
$G_L \backslash M_L$ according to
$$\{i,j,k\} \cdot \mu
=\{i \cdot \mu,j \cdot \mu,k \cdot \mu\},$$
for $i,j,k \in 
G_L \backslash M_L$ and $\mu \in M_L$.

There are 
$\binom{12}{3}=220$
triples of right cosets of
$G_L \backslash M_L$. 
The final proposition of this section 
describes the orbits
of the action of $M_L$ on those triples.
It will be used in the next section 
in the classification of
the three-term relations satisfied 
by the $L$ function. 

\begin{Proposition}
\label{P660}

There are two orbits of the action of $M_L$
on triples of cosets of
$G_L \backslash M_L$. One orbit is of length
160 and consists of all triples that are
$L$-coherent. The other orbit is of length 60 
and consists
of all triples that are $L$-incoherent.

\end{Proposition}

\begin{proof}

From Proposition \ref{P640},
it is easily seen that there are two orbits
given by the $L$-coherent and $L$-incoherent
triples.
A simple counting argument  
establishes the lengths
of the two orbits.

\end{proof}

\section{Three-term relations}

For every $i \in G_L \backslash M_L$,
we define the function
$$L_{i}(\vec{x})=L(\mu\vec{x}),$$
where $\vec{x}=(a,b,c,d,e,f,g)^{T} \in V$
and $\mu$ is any representative of
the right coset $i$. 
A three-term relation involving the functions
$L_i(\vec{x}), L_j(\vec{x})$, and $L_k(\vec{x})$
is said to be of type $\{i,j,k\}$.

\begin{Proposition}
\label{P710}

A three-term relation of 
the $L$-coherent type
$\{6,5,4\}$ is given by
\begin{align}
\label{710}
&\frac{\sin \pi (f-g)}{\Gamma(e-a)\Gamma(e-b)
\Gamma(e-c)\Gamma(e-d)}
L\left[
{\displaystyle a,b,c,d;
\atop
\displaystyle e;f,g}\right] \nonumber \\
+&\frac{\sin \pi (g-e)}{\Gamma(f-a)\Gamma(f-b)
\Gamma(f-c)\Gamma(f-d)}
L\left[
{\displaystyle a,b,c,d;
\atop
\displaystyle f;e,g}\right] \nonumber \\
+&\frac{\sin \pi (e-f)}{\Gamma(g-a)\Gamma(g-b)
\Gamma(g-c)\Gamma(g-d)}
L\left[
{\displaystyle a,b,c,d;
\atop
\displaystyle g;f,e}\right] 
=0.
\end{align}

\end{Proposition}

\begin{proof}

We have
\begin{eqnarray*}
&&\frac{1}{2 \pi i} \int_t
\frac{\Gamma(a+t)\Gamma(b+t)\Gamma(c+t)
\Gamma(d+t)\Gamma(1-e-t)\Gamma(1-f-t)
\Gamma(-t)}{\Gamma(g+t)}\\
&&\cdot [
\sin \pi e \sin \pi (1-f-t)
-\sin \pi f \sin \pi (1-e-t)
-\sin \pi (f-e) \sin \pi (-t)]dt\\
&&=0,
\end{eqnarray*}
since the quantity in square brackets in the above
integral is equal to zero, which
can be seen by applying
elementary trigonometric identities.
We break up the integral into three parts and 
use (\ref{270}) to simplify the result to
obtain
\begin{align*}
&\frac{1}{2 \pi i} \int_t
\frac{\sin \pi e 
\Gamma(a+t)\Gamma(b+t)\Gamma(c+t)
\Gamma(d+t)\Gamma(1-e-t)
\Gamma(-t)}{\Gamma(f+t)\Gamma(g+t)}dt\\
-&\frac{1}{2 \pi i} \int_t
\frac{\sin \pi f 
\Gamma(a+t)\Gamma(b+t)\Gamma(c+t)
\Gamma(d+t)\Gamma(1-f-t)
\Gamma(-t)}{\Gamma(e+t)\Gamma(g+t)}dt\\
-&\frac{1}{2 \pi i} \int_t
\frac{\left[
{\displaystyle
\sin \pi (f-e) 
\Gamma(a+t)\Gamma(b+t)\Gamma(c+t)
\atop
\displaystyle
\cdot
\Gamma(d+t)\Gamma(1-e-t)\Gamma(1-f-t)}\right]}
{\Gamma(1+t)\Gamma(g+t)}dt
=0.
\end{align*}

After the substitution 
$t \mapsto t+1-e$ in the third integral,
we express
each of the three integrals 
as an $L$ function
according to (\ref{310}).
The end result is a three-term relation of
type $\{6,5,\bar{4}\}$ which can be written as
\begin{align}
\label{720}
&\frac{\sin \pi e}{\Gamma(1+a-f)\Gamma(1+b-f)
\Gamma(1+c-f)\Gamma(1+d-f)}
L\left[
{\displaystyle a,b,c,d;
\atop
\displaystyle
e;f,g}\right] \nonumber \\
+&\frac{\sin \pi (-f)}{\Gamma(1+a-e)\Gamma(1+b-e)
\Gamma(1+c-e)\Gamma(1+d-e)}
L\left[
{\displaystyle a,b,c,d;
\atop
\displaystyle f;e,g}\right] \nonumber \\
+&\frac{\sin \pi (e-f)}{\Gamma(a)\Gamma(b)
\Gamma(c)\Gamma(d)}
L\left[
{\displaystyle 1-a,1-b,1-c,1-d;
\atop
\displaystyle 2-g;2-f,2-e}\right]
=0.
\end{align}

If we let
$$\mu=(14)(23)[(123)A]^3 \in M_L,$$
we have
$$\{6,5,\bar{4}\} \cdot \mu
= \{6,5,4\}.$$
Applying the change of variable
$\vec{x} \mapsto \mu \vec{x}$
to all terms and coefficients
in the relation (\ref{720}) 
yields the result.

\end{proof}

If $\vec{x}=(a,b,c,d,e,f,g)^T \in V$,
we define $\gamma_1(\vec{x}), \gamma_2(\vec{x})$,
and $\gamma_3(\vec{x})$ to be the respective
coefficients in front of the three $L$
functions in (\ref{710}).
This way the three-term relation 
(\ref{710}) can be written as
\begin{equation}
\label{730}
\gamma_1(\vec{x})L_{6}(\vec{x})
+\gamma_2(\vec{x})L_{5}(\vec{x})
+\gamma_3(\vec{x})L_{4}(\vec{x})
=0,
\end{equation}
or, equivalently, as
\begin{equation}
\label{732}
\sum_{i=0}^{2}\gamma_1((576)^i \vec{x})L((576)^i \vec{x})
=0.
\end{equation}

Let $\{i,j,k\}$ be any $L$-coherent triple. Since,
by Proposition \ref{P660}, 
$\{i,j,k\}$ is in the same orbit as $\{6,5,4\}$,
there exists $\mu \in M_L$ such that
$$\{6,5,4\} \cdot \mu = \{i,j,k\}.$$
Then a three-term relation of type $\{i,j,k\}$
is given by
\begin{equation}
\label{735}
\gamma_1(\mu\vec{x})L_{i}(\vec{x})
+\gamma_2(\mu\vec{x})L_{j}(\vec{x})
+\gamma_3(\mu\vec{x})L_{k}(\vec{x})
=0,
\end{equation}
or, equivalently, by
\begin{equation}
\label{737}
\sum_{i=0}^{2}\gamma_1((576)^i \mu \vec{x})
L((576)^i \mu \vec{x})
=0.
\end{equation}

\begin{Proposition}
\label{P720}

A three-term relation of the $L$-incoherent
type $\{6,5,\bar{6}\}$
is given by
\begin{align}
\label{740}
&\frac{\sin \pi g \sin \pi (f-g)}
{\left[{\displaystyle \Gamma(1+a-f)\Gamma(1+b-f)
\Gamma(1+c-f)
\atop
\displaystyle
\cdot\Gamma(1+d-f)
\Gamma(e-a)\Gamma(e-b)\Gamma(e-c)\Gamma(e-d)}
\right]} \nonumber \\
\cdot &L\left[
{\displaystyle a,b,c,d;
\atop
\displaystyle e;f,g}\right] \nonumber \\
+&\frac{1}{\pi^4}\left(\left[
{\displaystyle \sin \pi g \sin \pi (g-e) \sin \pi (f-a)
\atop
\displaystyle\cdot
\sin \pi (f-b) \sin \pi (f-c) \sin \pi (f-d)}\right]
\right. \nonumber \\
&\quad +\left.\left[
{\displaystyle \sin \pi f \sin \pi (e-f) \sin \pi (g-a)
\atop
\displaystyle \cdot
\sin \pi (g-b) \sin \pi (g-c) \sin \pi (g-d)}
\right]\right) \nonumber \\
\cdot &L\left[
{\displaystyle a,b,c,d;
\atop
\displaystyle f;e,g}\right] \nonumber \\
+&\frac{\sin \pi (e-f) \sin \pi (f-g)}
{\Gamma(a)\Gamma(b)\Gamma(c)\Gamma(d)
\Gamma(g-a)\Gamma(g-b)\Gamma(g-c)\Gamma(g-d)} \nonumber \\
\cdot &L\left[
{\displaystyle 1-a,1-b,1-c,1-d;
\atop
\displaystyle 2-e;2-f,2-g}\right]
=0.
\end{align}

\end{Proposition}

\begin{proof}

Let
$$\mu=(14)(23)[(123)A]^3(576) \in M_L.$$
We have
$$\{6,5,4\} \cdot \mu 
=\{5,4,\bar{6}\}.$$
This and (\ref{730}) 
lead to a three-term relation
of type $\{5,4,\bar{6}\}$ given by
$$\gamma_1(\mu\vec{x})L_{5}(\vec{x})
+\gamma_2(\mu\vec{x})L_{4}(\vec{x})
+\gamma_3(\mu\vec{x})L_{\bar{6}}(\vec{x})
=0.$$
We combine the above three-term relation
of type $\{5,4,\bar{6}\}$
with the three-term
relation of type $\{6,5,4\}$ given in
(\ref{730}) and cancel the terms involving
the function $L_4(\vec{x})$ to obtain
(\ref{740}). 

\end{proof}

If $\vec{x}=(a,b,c,d,e,f,g)^T \in V$,
we define $\beta_1(\vec{x}), \beta_2(\vec{x})$,
and $\beta_3(\vec{x})$ to be the respective
coefficients in front of the three $L$
functions in (\ref{740}).
This way the three-term relation 
(\ref{740}) can be written as
\begin{equation}
\label{750}
\beta_1(\vec{x})L_{6}(\vec{x})
+\beta_2(\vec{x})L_{5}(\vec{x})
+\beta_3(\vec{x})L_{\bar{6}}(\vec{x})
=0.
\end{equation}

Let $\{i,j,\bar{i}\}$ be any triple
that is $L$-incoherent. By Proposition
\ref{P660}, 
$\{i,j,\bar{i}\}$ is in the same orbit as 
$\{6,5,\bar{6}\}$ and so
there exists $\mu \in M_L$ such that
$$\{6,5,\bar{6}\} \cdot \mu = \{i,j,\bar{i}\}.$$
Then a three-term relation of type 
$\{i,j,\bar{i}\}$
is given by
\begin{equation}
\label{755}
\beta_1(\mu\vec{x})L_{i}(\vec{x})
+\beta_2(\mu\vec{x})L_{j}(\vec{x})
+\beta_3(\mu\vec{x})L_{\bar{i}}(\vec{x})
=0.
\end{equation}

To summarize our results, we have shown that
for every $\sigma_1,\sigma_2,\sigma_3 \in M_L$
such that $\sigma_1,\sigma_2,\sigma_3$ are in
different right cosets of $G_L$ in $M_L$,
there exists a three-term relation involving
the functions
$L(\sigma_1\vec{x}),L(\sigma_2\vec{x})$,
and $L(\sigma_3\vec{x})$.
We have a total of 220 three-term relations.
The three-term relations fall into two 
families based on $L$-coherence. The two fundamental
three-term relations for $L$-coherent
and for $L$-incoherent triples are given in
(\ref{730}) and (\ref{750}) respectively.
Any other three-term relation can be obtained from
one of those two through a change of variable
of the form 
$\vec{x} \mapsto \mu \vec{x}$ applied 
to all terms and coefficients. 
The result is a three-term relation
of the form (\ref{735}) or (\ref{755}). The 
appropriate matrix
$\mu \in M_L$ can be found through Proposition
\ref{P630}, 
which describes the actions of the generators
$a_1,a_2,a_3,a_4,a_5,a_{1'}$ of the group $M_L$.


\begin{thebibliography}{99}

\bibitem{Ba1} W.N. Bailey, Transformations
of well-poised hypergeometric series,
Proc. London Math. Soc. 
{\bf 36}
(1934), no.\ 2, 235--240.

\bibitem{Ba} W.N. Bailey, Generalized Hypergeometric 
Series, Cambridge University Press,
Cambridge, 1935.

\bibitem{Bar1} E.W. Barnes, A new development
of the theory of hypergeometric functions, Proc. London
Math. Soc. {\bf 2} (1908), no.\ 6, 141--177.

\bibitem{Bar2} E.W. Barnes, A transformation of
generalized hypergeometric series, Quart. J. of Math.
{\bf 41} (1910), 136--140.

\bibitem{BLS} W.A. Beyer, J.D. Louck, P.R. Stein,
Group theoretical basis of some identities for the
generalized hypergeometric series, J. Math. Phys. 
{\bf 28} (1987), no.\ 3, 497--508.

\bibitem{BRS} F.J. van de Bult, E.M. Rains, 
J.V. Stokman, Properties of generalized 
univariate hypergeometric functions,
Comm. Math. Phys. {\bf 275} (2007), 
no. 1, 37--95.

\bibitem{Bu} D. Bump, Barnes' second lemma and its
application to Rankin--Selberg convolutions,
Amer. J. Math. {\bf 109} (1987),
179--186.

\bibitem{D} G. Drake (ed.), Springer Handbook of Atomic,
Molecular and Optical Physics, Springer, New York, 2006.

\bibitem{DF} D.S. Dummit, R.M. Foote, Abstract
Algebra, Second Edition, John Wiley \& Sons, Inc.,
New York, 1999.

\bibitem{FGS} M. Formichella, R. Green, E. Stade,
Coxeter group actions on ${}_4F_3(1)$
hypergeometric series, 
Ramanujan J. (to appear).

\bibitem{Ga} C.F. Gauss, Disquisitiones generales
circa seriem infinitam $1+\frac{\alpha\beta}{1 \cdot \gamma}x
+\frac{\alpha(\alpha+1)\beta(\beta+1)}{1 \cdot 2 \cdot
\gamma(\gamma+1)}xx+$etc., Werke 3, K\"onigliche
Gesellschaft der Wissenschaften, G\"ottingen, 1876,
pp. 123--162.

\bibitem{Groe} W. Groenevelt, The Wilson
function transform, Int. Math. Res. Not.
{\bf 2003,} no. 52, 2779--2817.

\bibitem{Gr} A. Grozin, Lectures on QED and QCD:
Practical Calculation and Renormalization of One- and
Multi-Loop Feynman Diagrams, World Scientific,
Singapore, 2007.

\bibitem{Har} G.H. Hardy, Ramanujan:
Twelve Lectures on Subjects Suggested by
His Life and Work, Cambridge University Press,
Cambridge, 1940.

\bibitem{H} J.E. Humphreys, Reflection Groups and
Coxeter Groups, Cambridge University Press,
Cambridge, 1990.

\bibitem{KR} C. Krattenthaler, T. Rivoal,
How can we escape Thomae's relations?,
J. Math. Soc. Japan {\bf 58} (2006), no.\ 1,
183--210.

\bibitem{LJ1} S. Lievens, J. Van der Jeugt, Invariance
groups of three term transformations for basic
hypergeometric series, J. Comput. 
Appl. Math. {\bf 197} (2006), 1--14.

\bibitem{LJ2} S. Lievens, J. Van der Jeugt, Symmetry
groups of Bailey's transformations for 
${}_{10}\phi_9$-series, J. Comput. 
Appl. Math. {\bf 206} (2007), 498--519.

\bibitem{M} I. Mishev, Coxeter group actions on 
supplementary
pairs of Saalsch\"utzian ${}_4F_3(1)$
hypergeometric series, Ph.D. Thesis, University
of Colorado (2009).

\bibitem{R} J. Raynal, On the definition and
properties of generalized 6-$j$ symbols,
J. Math. Phys. 
{\bf 20} (1979), no.\ 12, 2398--2415.

\bibitem{Rao} K. Srinivasa Rao, J. Van der Jeugt,
J. Raynal, R. Jagannathan, V. Rajeswari,
Group theoretical basis for the terminating
${}_3F_2(1)$ series, J. Phys. A {\bf 25} (1992), 
no.\ 4, 861--876.

\bibitem{St1} E. Stade, Hypergeometric series and Euler
factors at infinity for $L$-functions on
$GL(3,\mathbb{R}) \times GL(3,\mathbb{R})$, Amer. J.
Math. {\bf 115} (1993), no. 2, 371--387.

\bibitem{St2} E. Stade, Mellin transforms of Whittaker
functions on $GL(4,\mathbb{R})$ and $GL(4,\mathbb{C})$,
Manuscripta Math. {\bf 87} (1995), 511--526.

\bibitem{St3} E. Stade, Mellin transforms of
$GL(n,\mathbb{R})$ Whittaker functions, Amer. J.
Math. {\bf 123} (2001), 121--161.

\bibitem{St4} E. Stade, Archimedean $L$-factors on
$GL(n) \times GL(n)$ and generalized Barnes integral,
Israel J. Math. {\bf 127} (2002), 201--220.

\bibitem{ST} E. Stade, J. Taggart, Hypergeometric 
series, a Barnes-type lemma, and Whittaker functions,
J. London Math. Soc. (2) {\bf 61} (2000), no. 1,
177--196.

\bibitem{T} J. Thomae, Ueber die Funktionen welche
durch Reihen der Form dargestellt werden:
$1+\frac{pp'p''}{1q'q''}+ \cdots$, J. Reine Angew.
Math. {\bf 87} (1879), 26--73.

\bibitem{Titch} E.C. Titchmarsh, The Theory of Functions, 
Oxford University Press,
London, 1952.

\bibitem{V} J. Van der Jeugt, K. Srinivasa Rao,
Invariance groups of transformations of
basic hypergeometric series, J. Math. Phys.
{\bf 40} (1999), no.\ 12, 6692--6700.

\bibitem{Wh1} F.J.W. Whipple, A group of generalized
hypergeometric series: relations between 120 allied series
of the type $F(a,b,c;e,f)$, Proc. London Math. Soc. {\bf 23}
(1925), no.\ 2, 247--263.

\bibitem{Wh3} F.J.W. Whipple, Well-poised series and other
generalized hypergeometric series, Proc. London Math. Soc.
{\bf 25} (1926), no.\ 2, 525--544.

\bibitem{Wh2} F.J.W. Whipple, Relations between
well-poised hypergeometric series of the type
${}_7F_6$, Proc. London Math. Soc. {\bf 40} (1936),
no.\ 2, 336--344.

\bibitem{WW}  E.T. Whittaker, G.N. Watson, A Course of Modern
Analysis, Cambridge University Press, Cambridge, 1963.

\end{thebibliography}
\end{document}